\documentclass[10pt]{amsart}
\usepackage[normalem]{ulem}

\usepackage{amsmath}
\usepackage{amsfonts}
\usepackage{amssymb}
\usepackage{mathrsfs}
\usepackage{amsthm}
\usepackage{amsmath,amssymb,amsfonts,amsthm,color,latexsym}
\usepackage{xcolor}
\usepackage{float}

\theoremstyle{plain}

\newcommand{\R}{\mathbb{R}}
\newcommand{\C}{\mathbb{C}}

\newcommand{\W}{\mathcal{W}}
\newcommand{\F}{\mathcal{F}}

\newcommand{\z}{\overline{z}}

\newcommand{\PP}{\mathbb{P}}
\newcommand{\sym}{\mathrm{Sym}}

\newcommand{\sing}{\mathrm{Sing}}
\newcommand{\cs}{\text{CS}}

\newtheorem{maintheorem}{Theorem}

\newtheorem{secondtheorem}{Theorem}

\newtheorem{thirdtheorem}{Theorem}

\newtheorem{theorem}{Theorem}[section]

\newtheorem{lemma}[theorem]{Lemma}
\newtheorem{proposition}[theorem]{Proposition}

\theoremstyle{definition}
\newtheorem{definition}{Definition}[section]
\newtheorem{example}{Example}[section]

\usepackage{graphicx}
\begin{document}
%\linenumbers
\title[Levi-flat hypersurfaces and holomorphic Webs]{On real-analytic Levi-flat hypersurfaces and holomorphic Webs}

\author{Ayane Adelina da Silva}
\address[Ayane Adelina da Silva]{Departamento de Matem\'atica, Universidade Federal de Minas Gerais, UFMG}
\curraddr{Av. Ant\^onio Carlos 6627, 31270-901, Belo Horizonte-MG, Brazil.}
\email{ayaneadelina@gmail.com}

\author{Arturo Fern\'andez-P\'erez}
\address[Arturo Fern\'andez-P\'erez]{Departamento de Matem\'atica, Universidade Federal de Minas Gerais, UFMG}
\curraddr{Av. Ant\^onio Carlos 6627, 31270-901, Belo Horizonte-MG, Brazil.}
\email{fernandez@ufmg.br}
\thanks{The second-named author is partially supported by CNPq-Brazil Grant Number 302790/2019-5 and  Pronex-Faperj.}
%\thanks{This work is supported by CNPq Brazil grant number 427388/2016-3 and Concytec Per\'u.}


\begin{abstract}
We investigate holomorphic webs tangent to real-analytic Levi-flat 
hypersurfaces on compact complex surfaces. Under certain conditions, we prove that a holomorphic web tangent to a real-analytic Levi-flat hypersurface admits a multiple-valued meromorphic first integral. We also prove that the Levi foliation of a Levi-flat hypersurface induced by an irreducible real-analytic curve in the Grassmannian $G(n+1,n)$ extends to an algebraic web on the complex projective space. 
\end{abstract}

\maketitle
%\tableofcontents
\section{Introduction}
\par The study of webs began in 1926-27 by W. Blaschke and G. Thomsen, see for instance \cite{Bol} and \cite{Thomsen}, and since then, webs theory has received important contributions from mathematicians such as Akivis-Shelekhov \cite{Akivis}, Goldberg \cite{goldberg}, Chern \cite{chern}, Griffiths \cite{griffiths}, Nakai \cite{nakai}, Tr\'epreau \cite{Trepreau}, Cavalier-Lehmann \cite{Lehmann}, \cite{cavalier} and Pereira-Pirio  \cite{JV} among others. The reader who wishes to get an outline of the history of Web Geometry may consult the book \cite[Appendix p.155]{JV}. Singular real-analytic Levi-flat hypersurfaces in complex manifolds arise naturally as invariant sets of holomorphic foliations and implicit differential equations. Such hypersurfaces were studied by Burns-Gong \cite{burns}, they have initiated a systematic study of singular real-analytic Levi-flat hypersurfaces with quadratic singularities and also have proved the existence of normal forms of quadratic Levi-flat hypersurfaces. 
\par This paper is devoted to the study of holomorphic webs tangent to
singular real-analytic Levi-flat hypersurfaces on compact complex manifolds. Initially, we will restrict our attention to real-analytic Levi-flat hypersurfaces and holomorphic webs on compact complex surfaces (complex manifolds of complex dimension two), and in the last part of the paper, we consider the \textit{holomorphic extension problem} of the Levi foliation of a real-analytic Levi-flat on the complex projective space $\mathbb{P}^n$, $n\geq 2$. Our goal is to obtain properties on the singular set of the Levi-flat hypersurface and as a consequence some properties of the tangent web, for instance, such a web admitting a multiple-valued meromorphic first integral, see Theorems \ref{first} and \ref{second2}. 
\par The first motivation to study webs tangent to Levi-flat hypersurfaces is the known result for foliations (1-webs) tangent to Levi-flat hypersurfaces established by Cerveau-Lins Neto in \cite[Theorem 1]{alcides}. They proved that any germ of codimension one holomorphic foliation tangent to Levi-flat hypersurface admits a local meromorphic first integral, that is, there exists a local meromorphic (possibly holomorphic) function $h$ such that the leaves of the foliation are contained in the level sets of $h$. Moreover, they have conjectured that for webs something similar occurs. In the case of local holomorphic webs, under some conditions, a result in favor of this conjecture was given in \cite[Theorem 1]{arturo}.
\par The second motivation is the \textit{holomorphic extension problem} of the Levi foliation of a Levi-flat hypersurface. In the local context, this problem has been partially solved in \cite[Theorem 1]{generic} and recently generalized by Shafikov-Sukhov \cite[Theorem 1.1]{Russians}. They proved that if a real-analytic Levi-flat hypersurface $M$ has a non-dicritical singularity at $p$ or if $M$ is a real-algebraic hypersurface, then the Levi foliation of $M$ extends to a holomorphic web admitting a multiple-valued meromorphic first integral. In Section \ref{projective-space}, we consider this problem for a special family of Levi-flat hypersurfaces, those defined by real-analytic curves in the Grassmannian $G(n+1,n)$, see for instance Theorem \ref{third}. 
\par The paper is organized as follows: in Section \ref{foli}, we define the concept of codimension one holomorphic foliations, first integrals, and the Camacho-Sad index of a foliation with respect to an invariant curve. Section \ref{Levi} is devoted to the definition of real analytic Levi-flat hypersurfaces in complex manifolds, Segre varieties associated with Levi-flat hypersurfaces, and dicritical singularity of a Levi-flat hypersurface. In Section \ref{webs-theory}, we define global webs as the closure of meromorphic multi-sections. We also define the concept de dicriticity of a web, this definition is from Cavalier-Lehmann \cite[D\'efinition 3.9]{Lehmann} and the notion of meromorphic first integral for webs given by Pan \cite[D\'efinition 4.1]{Pan}. We study the lifting of Levi-flat hypersurfaces to the projectivised cotangent bundle in Section \ref{levantamento}. In Sections \ref{primer} and \ref{segundo}, we will impose some conditions to holomorphic webs tangent to Levi-flat hypersurfaces in compact complex surfaces admit multiple-valued meromorphic first integrals, see for instance Theorems \ref{first} and \ref{second2}.
Some examples are presented in these sections. 
Finally, in Section \ref{projective-space}, we will solve the holomorphic extension problem of the Levi foliation of a Levi-flat hypersurface induced by an irreducible real-analytic curve in the Grassmannian $G(n+1,n)$. In this situation, we prove that the Levi foliation extends to an algebraic web in $\mathbb{P}^n$. 
\par This article is a partial compilation of the results of the Ph.D. Thesis of the first author (cf. \cite{ayane}) written under the supervision of the second author in the Department of Mathematics of the Federal University of Minas Gerais - Brazil.\\

\noindent {\bf Acknowledgments.}
 The authors are grateful to Federal University of Minas Gerais - Brazil and CAPES - Brazil for the support. 

\section{Holomorphic foliations}\label{foli}
\par Let $\{U_j\}_{j\in I}$ be an open covering of $X$. A \textit{codimension one holomorphic foliation} $\F$ on $X$ can be described by a collection of holomorphic 1-forms $\omega_j\in\Omega^{1}_{X}(U_{j})$ such that $\omega_j\wedge d\omega_j=0$ for all $j\in I$ and
\begin{equation*}
\omega_i=g_{ij}\omega_j\,\,\,\,\,\,\,\,\,\text{on}\,\,\,U_i\cap U_j,\,\,\,\,\,\,\,\,\,\,\,\,\,\,g_{ij}\in\mathcal{O}^{*}_{X}(U_i\cap U_j).
\end{equation*}
The cocycle $\{g_{ij}\}$ defines a line bundle $N_{\F}$ on $X$, called \textit{normal bundle} of $\F$.  The \textit{singular set} $\sing(\F)$ of $\F$ is the complex subvariety of $X$ defined by
$$\sing(\F)\cap U_{j}=\text{zeros of}\,\,\,\omega_{j},\,\,\,\,\,\,\,\,\,\,\,\,\,\,\,\forall j\in I.$$
We always consider saturated foliations, i.e., the zero set of every $\omega_j$ has codimension at least two. Therefore, by definition, $\sing(\F)$ has codimension at least two. A point $q\not\in\sing(\F)$ is said to be \textit{regular} of $\F$. A complex hypersurface $H\subset X$ is said to be \textit{invariant} by $\F$ if all connected components of $H$ are leaves of $\F$. 
We say that $\mathcal{F}$ admits a \textit{meromorphic} (\textit{holomorphic}) first integral at $p\in X$, if there exists a neighborhood $U\subset X$ of $p$ and  
a \textit{meromorphic} (\textit{holomorphic}) function $h$ defined in $U$ such that its indeterminacy (zeros)
set is contained in $\sing(\F)\cap U$ and its level curves contain the leaves of $\mathcal{F}$ in $U$. When $X$ is a complex surface, a point singular $p\in\sing(\F)$ of a holomorphic foliation $\F$ is called \textit{dicritical} if for every neighborhood $U$ of $p$, infinitely many leaves have $p$ in their closure. Otherwise it is called \textit{non-dicritical}.
\par To continue, we state an important result on the Camacho-Sad index \cite{CS} of holomorphic foliations on complex surfaces: this index concerns the computation of the self-intersection $C\cdot C$, where $C\subset X$ is an invariant compact curve by $\mathcal{F}$. More references for index theorems we refer the reader to \cite{index} and \cite{birational}.
\subsection{Camacho-Sad index} 
Let $X$ be a complex surface and $\F$ be a holomorphic foliation on $X$.
Let us consider a \textit{separatrix} $C$ at $p\in X$, i.e., $C$ is an irreducible analytic curve invariant by $\F$. Here, \textit{invariant} means that if a point of $C$ belongs to the regular part of $\F$, then the whole leaf through this point is included in $C$.
Let $f$ be a holomorphic function on a neighborhood of $p$ such that $C =\{f=0\}$. 
 Then, according to \cite[Lemma 1.1]{suwa}, there are holomorphic functions $g$, $k$ and a holomorphic 1-form $\eta$ on a neighborhood of $p$ such that
$$g\omega=kdf+f\eta$$
and moreover $k$ and $f$ are prime, i.e., $k\neq 0$ on $C^{*}=C\setminus\{p\}$. 
The Camacho-Sad index \cite{CS} is defined as 
$$\cs_p(\F,C)=-\frac{1}{2\pi i}\int_{\partial{C}}\frac{\eta}{k},$$
where $\partial{C}=C\cap S^{3}$ and $S^3$ is a small sphere around $p$; $\partial{C}$ is oriented as  a boundary of $S^{3}\cap B^4$, with $B^4$ a ball containing $p$. 
 If $C\subset X$ is a compact complex curve invariant by $\F$, one has the formula due to Camacho-Sad. 
\begin{theorem}\cite[Camacho-Sad]{CS}\label{CS}
$$\sum_{p\in\sing(\F)\cap C}\cs_p(\F,C)=C\cdot C.$$
\end{theorem}
\section{Levi-flat hypersurfaces}\label{Levi}
Let $X$ be a complex manifold of complex dimension $n\geq 2$, a  closed  subset $M\subset X$ is a \textit{real-analytic subvariety}  if for every $p\in M$, there are real-analytic functions with real values $\varphi_1,\ldots,\varphi_k$ defined in a neighborhood $U\subset X$ of $p$, such that $M\cap U$ is equal to the set where all $\varphi_1,\ldots,\varphi_k$ vanish. A complex subvariety is precisely the same notion, considering holomorphic functions instead of real analytic functions.  We say that a real-analytic subvariety $M$ is \textit{irreducible} if whenever we write $M=M_1\cup M_2$ for two subvarieties $M_1$ and $M_2$ of $X$, then either $M_1=M$ or $M_2=M$. If $M$ is irreducible, it has a well-defined dimension $\dim_\R M$. Let $M_{reg}$ denote its \textit{regular part}, i.e., the subset of points near which $M$ is a real-analytic submanifold of dimension equal to $\dim_\R M$. A \emph{real-analytic hypersurface} is a real-analytic subvariety of real codimension one in $X$.
\par If $M \subset X$ is a real-analytic  submanifold of real codimension one, for each   $p \in M$, there is a unique complex hyperplane $L_{p}$ contained in the tangent space $T_{p}M \subset T_{p}X$. Then, it defines a real-analytic distribution $p \mapsto L_{p}$ of complex hyperplanes in $T M$.  When this distribution is \textit{integrable} in the sense of Frobenius, we say that $M$ is a {\em Levi-flat}. In this case, $M$ is foliated by   immersed complex manifolds of complex dimension $n-1$. This foliation, denoted by $\mathcal{L}_M$, is known as  {\em Levi foliation}. If $M$ is singular, let us denote by $\sing(M)$ the \textit{singular} locus of $M$, that is, the set of points where $M$ fails to be a real-analytic submanifold. In this case, $M$ is said to be \emph{Levi-flat} if its regular part $M_{reg}$ is Levi-flat.
\par A normal form for regular points of $M$ was given by
E. Cartan \cite[Theorem IV]{cartan}: at each $p \in M_{reg}$, there are holomorphic coordinates $(z_{1},\ldots,z_{n})$ in a neighborhood $U$ of $p$
such that
\begin{equation}\label{formalocal-hlf}
M_{reg} \cap U = \{\Im m(z_{n}) = 0\}.
\end{equation}
 As a consequence, the leaves of
$\mathcal{L}_M$
 have local equations $z_{n} = c$, for $c \in \mathbb{R}$. Cartan's local trivialization \cite{cartan} allows the extension of  $\mathcal{L}_M$  to a non-singular holomorphic foliation  in a neighborhood of  $M_{reg}$ in $X$, which is unique as a germ around $M_{reg}$. In general, it is not possible to extend    $\mathcal{L}_M$   to a singular holomorphic foliation in a neighborhood of $M$, for instance, there exists examples of Levi-flat hypersurfaces where its Levi foliation extends to 
holomorphic webs, see Brunella \cite{brunella}. 
\par When there is a singular codimension one holomorphic foliation $\F$ in the ambient space $X$ coinciding with the Levi foliation $\mathcal{L}_M$ on $M_{reg}$, we say either that $M$ is \emph{invariant} by $\F$ or that $\F$ is \emph{tangent} to $M$. Germs of codimension one holomorphic foliations tangent to germs of real-analytic Levi-flat hypersurfaces was studied by Cerveau-Lins Neto \cite[Theorem 1]{alcides}, they proved that a such foliation admits a non-constant meromorphic (possibly holomorphic) first integral.
\par In order to distinguish singularities of $M$ we have the following definition. 
\begin{definition}
 A singular point $p\in \sing(M)$ is called \textit{dicritical} if, for every neighborhood $U$ of $p$, infinitely many leaves of the Levi-foliation on $M_{reg}\cap U$ have $p$ in their closure. 
 \end{definition}
This definition is analogous to definition of dicritical singularities of foliations. To continue, we establish the relation between dicritical points and degenerate Segre points of $M$.
\subsection{Segre varieties}
Assume that $M$ is defined by $\{F=0\}$, where $F$ is real-analytic function in a neighborhood $U\subset\mathbb{C}^n$. Without loss of generality we may assume that $U$ is a polydisc centered at the origin of $\mathbb{C}^n$. Then $F$ admits a Taylor expansion
\begin{equation}\label{defi}
F(z,\bar{z})=\sum_{\mu,\nu}F_{\mu\nu}z^{\mu}\bar{z}^{\nu}
\end{equation}
where $\bar{F}_{\mu\nu}=F_{\nu\mu}$, $\mu=(\mu_{1},\ldots,\mu_{n})$, $\nu=(\nu_{1},\ldots,\nu_{n})$, 
$z^{\mu}=z_{1}^{\mu_{1}}\ldots z_{n}^{\mu_{n}}$ and $\bar{z}^{\nu}=\bar{z}_{1}^{\nu_{1}}\ldots \bar{z}_{n}^{\nu_{n}}$. The complexification of $F$ is given by 
\begin{equation}\label{defi2}
F_{\C}(z,w)=\sum_{\mu,\nu}F_{\mu\nu}z^{\mu}w^{\nu},
\end{equation}
that is, we replace the variable $\bar{z}$ with an independent variable $w$. We assume that the neighborhood $U$ is chosen so small that the series (\ref{defi2}) converges for all $z,w\in U$. Therefore, $F_{\C}(z,\bar{w})$ is holomorphic in $z\in U$ and antiholomorphic in $w\in U$. Fix $p\in U$, the \textit{Segre variety} associated to $M$ at $p$ is the complex hypersurface  defined by
\begin{equation}\label{segre-variety}
Q_{p}:=\{z\in U:F_{\C}(z,\bar{p})=0\}.
\end{equation}
It is easy to check that Segre varieties are defined invariantly with respect to the choice of the defining function $F(z,\bar{z})$ of $M$. Now assume that $M$ is Levi-flat and denote by $L_{p}$ the leaf of $\mathcal{L}_M$ through $p\in M_{reg}$, an interesting property is that $L_p$ is contained in the unique irreducible component $S_p$ of $Q_p$. For more details about Segre varieties we refer the reader to \cite[Corollary 2.2]{Russians}. 
\par A singular point $p$ is called \textit{Segre degenerate} if $\dim Q_p=n$. In \cite{pinchuk}, Pinchuk-Shafikov-Sukhov proved the following characterization of dicritical singularities of a Levi-flat hypersurface.
\begin{theorem}
Let $M$ be a germ of an irreducible germ of a real-analytic Levi-flat hypersurface at $p\in\C^n$ and $p\in\overline{M_{reg}}$. Then $p$ is a dicritical point if and only if it is Segre degenerate.
\end{theorem}

\section{Codimension one holomorphic webs}\label{webs-theory}
In this section, we follow the references of \cite{cavalier} and \cite{JV}.
A germ at $0\in\mathbb{C}^{n}$, $n\geq 2$, of codimension one $d$-web $\mathcal{W}$ is an equivalence class $[\omega]$ of germs of $d$-symmetric 1-forms, these class are sections of $\sym^{d}\Omega^{1}(\mathbb{C}^{n},0)$, modulo multiplication by $\mathcal{O}^{*}(\mathbb{C}^{n},0)$ such that a suitable representative $\omega$ defined in a connected neighborhood $U$ of the origin satisfies the following conditions:
\begin{enumerate}
\item[(a)] the zero set of $\omega$ has codimension at least two; 
\item[(b)] $\omega$ is a homogeneous polynomial of degree $d$ in the ring $\mathcal{O}_{n}[dx_{1},\ldots,dx_{n}]$, and square-free.
\item[(c)] (Brill's condition) For a generic $\zeta\in U$, $\omega(\zeta)$ is a product of $d$ linear forms.
\item[(d)] (Frobenius's condition) For a generic $\zeta\in U$, the germ of $\omega$ at $\zeta$ is the product of $d$ germs of integrable 1-forms.
\end{enumerate}
\par Both conditions $(c)$ and $(d)$ are automatic for germs at $0\in\mathbb{C}^{2}$ of webs and non-trivial for germs at $0\in\mathbb{C}^{n}$ when $n\geq 3$. 

\subsection{Global webs}\label{webs_theory}
A global holomorphic $d$-web $\W$ on a complex manifold $X$ is given by an open covering $\mathcal{U}=\{U_j\}_{j\in J}$ of $X$ and $d$-symmetric 1-forms $\omega_j\in \sym^{d}\Omega^{1}_X(U_j)$ subject to the conditions:
\begin{enumerate}
\item for each non-empty intersection $U_i\cap U_j$ of elements of $\mathcal{U}$ there exists a non-vanishing function $g_{ij}\in\mathcal{O}^{*}(U_i\cap U_j)$ such that $$\omega_i=g_{ij}\omega_j;$$
\item for every $U_i\in\mathcal{U}$ and every $\zeta\in U_i$ the germification of $\omega_i$ at $\zeta$ satisfies the conditions $(a)$, $(b)$, $(c)$ and $(d)$, in other words, the germ of $\omega_i$ at $\zeta$ is a representative of a germ of a singular web. 
\end{enumerate}
\par The transition functions $g_{ij}$ determine a line-bundle $\mathcal{N}$ over $X$ and the $d$-symmetric 1-forms $\{\omega_j\}$ patch together to form a global section of $\sym^{d}\Omega^{1}_{X}\otimes\mathcal{N}$. The line-bundle $\mathcal{N}$ will be called the \textit{normal bundle} of $\W$. Two global sections $\omega,\omega'\in H^{0}(X,\sym^{d}\Omega^{1}_{X}\otimes\mathcal{N})$ determine the same web if and only if they differ by the multiplication by an element $g\in H^{0}(X,\mathcal{O}_X^{*})$. If $X$ is compact, then a global $d$-web is nothing more than an element of $\mathbb{P}H^{0}(X,\sym^{d}\Omega^{1}_{X}\otimes\mathcal{N})$ for a suitable line-bundle $\mathcal{N}\in Pic(X)$, with germification of any representative at any point of $X$ satisfying conditions $(a)$, $(b)$, $(c)$ and $(d)$. 
\par When $X$ is a complex variety for which every line-bundle has non-zero meromorphic sections one can alternatively define global $d$-webs as equivalence classes $[\omega]$ of meromorphic $d$-symmetric 1-forms modulo multiplication by meromorphic functions such that at a generic point $\zeta\in X$ the germification of any representative $\omega$ satisfies the very same conditions refereed to above. The transition to the previous definition is made by observing that  meromorphic $d$-symmetric 1-form $\omega$ can be interpreted as a global holomorphic section of $\sym^{d}\Omega_X^{1}\otimes\mathcal{O}_X((\omega)_{\infty}-(\omega)_0)$ where $(\omega)_0$, respectively $(\omega)_{\infty}$, stands for the zero divisor, respectively polar divisor, of $\omega$. 
\subsection{Webs as the closure of meromorphic multi-sections}
\par Let $X$ be a complex manifold, not necessarily compact, of complex dimension $n$. Let $TX$ be the complex tangent bundle, and $\tilde{X}$ the $(2n-1)$-dimensional manifold equal to the total space of the grassmanian bundle $\pi:\mathbb{G}_{n-1}X\to X$, that is, a point of $\tilde{X}$ over a point $\zeta\in X$ is a $(n-1)$-dimensional sub-vector space of $T_\zeta X$, and is called a \textit{contact element} of $X$ at $\zeta=\pi(\tilde{\zeta})$. In the sequel, indices such as $\lambda,\mu,\ldots$ will denote integers running from $1$ to $n$, while $\alpha,\beta,\ldots$ will denote integers running from $1$ to $n-1$. For all family $v=(v_{\alpha})$ of $n-1$ vectors linearly independent in $T_\zeta X$, $[v]$ or $[v_1,\ldots,v_{n-1}]$ or $[(v_{\alpha})_{\alpha}]$ will denote the contact element generated by $v$ in $\tilde{X}_\zeta$.
\par Since any $\tilde{\zeta}$ over $\zeta$ is the kernel of a non-vanishing 1-form on $T_\zeta X$ well defined up to multiplication by a scalar unit, we can also identify $\tilde{X}$ with the total space of $\mathbb{P}T^{*}X\to X$. 
\par Let $\mathcal{T}\subset \pi^{-1}(TX)$ be the tautological vector bundle over $\tilde{X}$, whose fiber over $[v_1,\ldots,v_{n-1}]$ is the sub-vector space of $T_\zeta X$ generated by $(v_1,\ldots,v_{n-1})$, identified with the line in $T^{*}_{\zeta} X$ of all 1-forms vanishing on the $v_{\alpha}$, for all $1\leq \alpha\leq n-1$. It follows from \cite[Lemma 2.1]{cavalier} that the dual $\mathcal{L}^{*}$  of the quotient bundle $\mathcal{L}=\pi^{-1}(TX)/\mathcal{T}$ is the tautological complex line bundle of $\mathbb{P}T^{*}X$. We have two exact sequences of vector bundles
\begin{equation}\label{L}
0\to\mathcal{T}\to\pi^{-1}(TX)\to\mathcal{L}\to 0,
\end{equation}
and 
\begin{equation}\label{V}
0\to\mathcal{V}\to T\tilde{X}\to\pi^{-1}(TX)\to 0,
\end{equation}
where $\mathcal{V}$ denotes the sub-bundle of tangent vectors to $\tilde{X}$ which are ``vertical", that is, tangent to the fiber of $\pi$. By composition of the projection $T\tilde{X}\to\pi^{-1}(TX)$ of (\ref{V}) with the projection $\pi^{-1}(TX)\to\mathcal{L}$ of (\ref{L}), we obtain a canonical holomorphic 1-form on $\tilde{X}$ with coefficients in $\mathcal{L}$
$$\omega:T\tilde{X}\to\mathcal{L},$$ which is called the \textit{tautological contact form}. 
\par Let $x=(x_1,\ldots,x_n)$ be local holomorphic coordinates on an open set $U$ of $X$, and $\zeta$ a point of $U$. Let $u=(u_1,\ldots,u_n)$ denote the coordinates in $T^{*}_{\zeta} X$ with respect to the basis $(dx_1)_{\zeta},\ldots,(dx_n)_{\zeta}$; 
let $[u]=[u_1,\ldots,u_n]$ denote the point in $\mathbb{P}T^{*}_{\zeta}X$ of homogeneous coordinates $u$ with respect to the same basis in $T^{*}_{\zeta} X$, and let $p=(p_1,\ldots,p_{n-1})$ be the system of affine coordinates on the affine subspace $u_{n}\neq 0$ of $\mathbb{P}T^{*}_\zeta X$ defined by $p_{\alpha}=-\frac{u_{\alpha}}{u_n}$, for all $1\leq\alpha\leq n-1$. Then, $(x,p)=(x_1,\ldots,x_n,p_1\ldots,p_{n-1})$
is a system of local holomorphic coordinates on the open set $\tilde{U}$ of contact elements above $U$ which are not parallel to $\left(\frac{\partial}{\partial{x}_{n}}\right)_{\zeta}$: the point of coordinates $(x,p)$ is the $(n-1)$-dimensional subspace of $T^{*}_{\zeta}X$ which is the kernel of the 1-form $$\eta=(dx_n)_{\zeta}-\sum_{\alpha=1}^{n-1}p_{\alpha}(dx_{\alpha})_{\zeta},$$
it is generated by the vectors fields $$(X_{\alpha})_{\zeta}=\left(\frac{\partial}{\partial{x}_{\alpha}}\right)_{\zeta}+p_{\alpha}\left(\frac{\partial}{\partial{x}_{n}}\right)_{\zeta},\,\,\,\,\,\forall\,\,\, 1\leq\alpha\leq n-1.$$
\par Of course, the vector fields $(X_{\alpha})_{\zeta}$, for all $1\leq\alpha\leq n-1$, define a holomorphic local trivialization $\sigma_{\mathcal{T}}$ of $\mathcal{T}$ over $\tilde{U}$, while the image $\sigma_{\mathcal{T}}=\left[\frac{\partial}{\partial{x}_{n}}\right]_{\zeta}$
of $\left(\frac{\partial}{\partial{x}_{n}}\right)_{\zeta}$ by the projection $\pi^{-1}(TX)\to\mathcal{L}$ of (\ref{L}) defines a holomorphic local trivialization of $\mathcal{L}$.
According to \cite[Lemma 2.4]{cavalier}, the local form $$\eta=(dx_n)_{\zeta}-\sum_{\alpha=1}^{n-1}p_{\alpha}(dx_{\alpha})_{\zeta}$$
is a contact form on $\tilde{U}$, and defines a local holomorphic trivialization of $\mathcal{L}^{*}$. Moreover, the trivializations $\eta$ and $\sigma_{\mathcal{T}}=\left[\frac{\partial}{\partial{x}_{n}}\right]_{\zeta}$ are dual to each other and the 1-form $\eta\otimes\sigma_{\mathcal{L}}$ is equal to the restriction of $\omega$ to $\tilde{U}$. 

\par Let $W$ be a complex subvariety  of $\tilde{X}$ having pure dimension $n$. Let $\pi_W:W\to X$ be the restriction to $W$ of the projection $\pi:\tilde{X}\to X$. Denote by $W'$ the \textit{regular part} of $W$, $\Gamma_W$ the set of points $\tilde{\zeta}\in W$ where either $W$ is singular, or the differential $d\pi_W:T_{\tilde{\zeta}}W'\to T_{\pi(\tilde{\zeta})}X$ is not an isomorphism and $W_0$ the complementary subset (included into $W'$) of $\Gamma_W$ in $W$, $X_0=\pi(W_0)$. 
\par Let $d\geq 1$ be an integer.
\begin{definition}
We shall say that $W$ is a \textit{holomorphic $d$-web} on $X$ if 
\begin{enumerate}
\item the map $\pi_W:W\to X$ is surjective,
\item the restriction $\omega_W$ to $W'$ of the tautological contact form $\omega:T\tilde{X}\to\mathcal{L}$ satisfies the integrability condition, that is, given a local trivialization $\sigma_{\mathcal{L}}$ of $\mathcal{L}$, the restriction $\eta_W$ to $W'\cap\tilde{U}$ of the local contact form $\eta$ such that $\omega=\eta\otimes\sigma_{\mathcal{L}}$ satisfies the condition $\eta_W\wedge d\eta_W=0$ (this condition does not depend on the local trivialization $\sigma_{\mathcal{L}}$). We denote by $\F_W$ the codimension one holomorphic foliation induced by $\eta_W$,
\item the restriction $\pi_W:W_0\to X_0$ of $\pi_W$ to $W_0$ is a $d$-fold covering, 
\item the complex analytic set $\Gamma_W$ has complex dimension at most $n-1$, or is empty,
\item for any $\zeta\in X$, $\pi_W^{-1}(\zeta)=W\cap \pi^{-1}(\zeta)$ is an algebraic subset of degree $d$ and  dimension $0$ in $\mathbb{P}T^{*}_\zeta X$. 
\end{enumerate}
\end{definition}
\par The complex analytic set $\Gamma_W$ is called \textit{critical set} of the web, its projection $\Delta=\pi(\Gamma_W)$ the \textit{caustic} or the \textit{singular part}, and is complementary part $X_0$ the \textit{regular part} of the web. 
We have the following remarks:
\begin{enumerate}
\item The condition $(1)$ in the above definition is automatic when $X$ is compact, because of $(3)$ and $(4)$.
\item Any holomorphic $d$-web $W$ on $X$ induces a $d$-web $W|_U=W\cap \pi^{-1}(U)$ on any open set $U$ of $M$.
\end{enumerate}
\begin{definition}
The irreducible components $C$ of the subvariety $W$ in $\tilde{X}$ are called the components of the web. They all are webs. The web is said to be \textit{irreducible} if it has only one component. It is said to be \textit{smooth} if $W$ is smooth ($W'=W$, which implies that it is irreducible), and \textit{quasi-smooth} if any irreducible component is smooth.  
\end{definition}
\subsection{The foliation associated to a web}
The condition $(2)$ of the definition of a web implies that $\omega_W:TW'\to\mathcal{L}$ defines a holomorphic foliation $\F_W$ on $W'$. This foliation may have singularities, but not on $W_0$ (it is defined by the contact 1-form that does not vanish). The projection $\pi_W$ of the restriction of this foliation to $W_0$ generates locally, near any point $p\in X_0$, $d$ distinct one-codimensional regular foliations $\F_i$, and $\F_W$ may be understood as a \textit{decrossing} of these $d$ foliations. 
\begin{definition}
We call \textit{leaf} of the web any complex hypersurface in $X$ whose intersection with $X_0$ is locally a leaf of one of the local foliations $\F_i$.  
\end{definition}
\par Reciprocally, given $d$ distinct regular holomorphic foliations $\F_i$ ($1\leq i\leq d$) of codimension one on some open subset $U$ of $M$, assume moreover that these foliations are mutually transverse to each over at any point of $U$, and that there exists holomorphic coordinates $x=(x_1,\ldots,x_n)$ on $U$ such that $\frac{\partial}{\partial{x}_n}$
be tangent to none of the foliations $\F_i$, we may define $\tilde{\F}_i$ by an integrable non-vanishing 1-form $$\eta_i=dx_n-\sum^{n-1}_{\alpha=1}p^{i}_{\alpha}dx_{\alpha}.$$ The submanifolds $U_i$ of $\tilde{X}$ defined by the $n-1$ equations $p_{\alpha}=p^{i}_{\alpha}(x)$ ($1\leq\alpha\leq n-1$) does not depend on the choice of the local coordinates. The union $W_U=\coprod_{i=1}^{d}U_i$ is then a $d$-fold trivial covering space of $U$, and defines a $d$-web (with no critical set). Denoting by $\pi_i:U_i\to U$
the restriction of $\pi$ to the sheet $U_i$, the form $\pi^{*}_i(\eta_i)$ is equal to the restriction of $d\tilde{x}_n-\sum^{n-1}_{\alpha=1}p_{\alpha}(\tilde{x})d\tilde{x}_{\alpha}$
 to $U_i$, where $\tilde{x}=(\tilde{x}_1,\ldots,\tilde{x}_n)$, and $\tilde{x}_\alpha$ is the restriction of $x_{\alpha}$ to $U_i$
\subsection{Dicriticity of a holomorphic web}
Since $TW_0$ is isomorphic to $\pi_W^{-1}(TX_0)$, the natural injection $\mathcal{T}\to\pi^{-1}(TX)$ defines an obvious injective map $\mathcal{T}|_{W_0}\to TW_0$, and the image of this map is the tangent bundle of the restriction foliation $\F_W|_{W_0}$ since it is annihilated by $\omega_W$, the foliation $\F_W$ is non-singular on $W_0$. 
\par Let $B$ be an irreducible component of $W$ and $B'$ its regular part. In general, $\F_W$ will have singularities on $B'$.
\begin{definition}\label{dicri}
We shall say that the web is \textit{dicritical} on $B$ if the restriction of the foliation $\F_W$ to $B'$ is non-singular. Otherwise, the web is \textit{non-dicritical} on $B$.
We shall say that the web is \textit{dicritical} if it is dicritical on every irreducible component.  
\end{definition} 
\par We remark that the above definition of dicriticity of a web has been introduced by Cavalier-Lehmann \cite[D\'efinition 3.9]{Lehmann}. Such definition is different from the definition of dicriticity of foliations, however, both definitions coincide when la web is a radial foliation, i.e., it is locally induced by the 1-form $\omega=xdy-ydx$ around a singular point. 
\subsection{Meromorphic first integral for webs}
Let $\W$ be a codimension one holomorphic web on $X$. We say that $\W$ has a \textit{multiple-valued meromorphic first integral} if for each irreducible factor (or component) $\W'$ of $\W$ there exists a polynomial $P=P_{\W'}\in\C(X)[t]$ such that outside an analytic subset $\triangle=\triangle_P\subset X$ any connected component of the analytic subset $\{P=t_0\}$ is a leaf of $\W'$ for $t_0$ generic in $\C$. The product of the polynomials $P_{\W'}$ with $\W'$ a factor of $\W$ is called a multiple-valued meromorphic first integral for $\W$.
%\section{Real Analytic Levi-flat hypersurfaces}
%\par Motived by \cite{Lehmann} and \cite{generic}, we study singular codimension one holomorphic webs tangent to singular real-analytic Levi-flat hypersurfaces in compact complex manifolds with emphasis on the type of singularities of them. 
%
%\par To prove the main result of this section, we need the following result. 
%\begin{theorem}[Cerveau-Lins Neto \cite{alcides}]\label{lins-cerveau}
% Let $\F$ be a germ of codimension one holomorphic foliation at $0\in\mathbb{C}^{n}$, $n\geq{2}$, tangent to a germ of an irreducible real-analytic hypersurface $M$. Then $\F$ has a non-constant meromorphic first integral. In the case of dimension two, we can precise more:
%\begin{enumerate}
%\item If $\F$ is dicritical then it has a non-constant meromorphic first integral.
%\item If $\F$ is non-dicritical then it has a non-constant holomorphic first integral.
%\end{enumerate}
%\end{theorem}

\section{Lifting of Levi-flat hypersurfaces to the projectivised cotangent bundle}\label{levantamento}
In this section we give some remarks about the lifting of a Levi-flat hypersurface $M\subset X$ to the projectivised cotangent bundle of $TX$. Let $\mathbb{P}T^{*}X$ be the projectivised cotangent bundle of $X$ and $M$ be an irreducible real analytic hypersurface Levi-flat in $X$. Let $\pi:\mathbb{P}T^{*}X\rightarrow X$ be the usual projection. Since $\mathbb{P}T^{*}X$ is a $\mathbb{P}^{n-1}$-bundle over $X$, whose fiber $\mathbb{P}T^{*}_{\zeta}X$ over $\zeta\in X$, the regular part $M_{reg}$ of $M$ can be lifted to $\mathbb{P}T^{*}X$: just take, for every
$\zeta\in M_{reg}$, the complex hyperplane 
\begin{equation}
T_{\zeta}^{\mathbb{C}}M_{reg}:=T_{\zeta}M_{reg}\cap i(T_{\zeta}M_{reg})\subset T_{\zeta}X. 
\end{equation}
We call 
\begin{equation}
M'\subset\mathbb{P}T^{*}X
\end{equation}
this lifting of $M_{reg}$. We remark that it is no more a hypersurface: its real dimension
$2n-1$ is half of the real dimension of $\mathbb{P}T^{*}X$. However, it is still ``Levi-flat'', in a
sense which will be precised below.

\par Take now a point $y$ in the closure $\overline{M'}$ projecting on $X$ to a point $x\in \overline{M_{reg}}$. Now, we shall consider the following results from Brunella \cite[Lemma 2.1 and Proposition 2.2]{brunella}. 
\begin{lemma}\label{brunella-lemma}
 There exist, in a germ of neighborhood $U_{y}\subset \mathbb{P}T^{*}X$ of $y$, a germ of real analytic subset $N_{y}$
of real dimension $2n-1$ containing $M'\cap U_{y}$.
\end{lemma}

\begin{proposition}\label{brunella-proposition}
 Under the above conditions, in a germ of neighborhood $V_{y}\subset U_{y}$ of $y$, there exists a germ of complex analytic subset $Y_{y}$ of complex dimension $n$ containing $N_{y}\cap V_{y}$.
\end{proposition}
\par As consequence, we have an application of the above results.
\begin{proposition}\label{lift}
Let $M$ be an irreducible real-analytic Levi-flat hypersurface in a complex manifold $X$ of complex dimension $n$. Then there exists a real analytic subvariety $N\subset\mathbb{P} T^{*}X$ of real dimension $2n-1$ which contains $M'$.
\end{proposition}
\begin{proof}
Let $y\in\overline{M'}$, it follows from Lemma \ref{brunella-lemma} that there exists, in a neighborhood $U_{y}\subset \mathbb{P}T^{*}X$ containing $y$, a real analytic subset $N_{y}$ of dimension $2n-1$ containing $M'\cap U_{y}$. Proposition \ref{brunella-proposition} implies that there exists, in a neighborhood $V_{y}\subset U_{y}$ of $y$, a complex analytic subset $Y_{y}\subset\mathbb{P}T^{*}X$ of complex dimension $n$ containing $N_{y}\cap V_{y}$. Then $N_{y}\cap V_{y}\subset Y_y$ is a real analytic hypersurface in $Y_{y}$, and it is Levi-flat because each irreducible component contains a Levi-flat piece (cf. \cite[Lemma 2.2]{burns}).
\par Let us denote $N'_{y}=N_{y}\cap V_{y}$. These local constructions are sufficiently 
canonical to be patched together, when $y$ varies on $\overline{M'}$: if $Y_{y_{1}}\subset V_{y_{1}}$ and $Y_{y_{2}}\subset V_{y_{2}}$ are as above,
 with $M'\cap V_{y_{1}}\cap V_{y_{2}}\neq \emptyset$, then $Y_{y_{2}}\cap( V_{y_{1}}\cap V_{y_{2}})$ and $Y_{y_{1}}\cap( V_{y_{1}}\cap V_{y_{2}})$ 
have some common irreducible components containing $M'\cap V_{y_{1}}\cap V_{y_{2}}$, so that $N'_{y_{1}}$,  
$N'_{y_{2}}$ can be glued by identifying those components.
 In this way, we obtain a real-analytic subvariety $N$ on $\mathbb{P}T^{*}X$ which contains $M'$.
\end{proof}
We need the following proposition of Lebl  \cite[Proposition 2.8]{lebl}.
\begin{proposition}\label{piece}
Let $M\subset\C^{n}$ be a connected real-analytic submanifold of real dimension $2n-1$. M is Levi-flat if and only if there exists an open set $N\subset M$ such that $N$ is Levi-flat.
\end{proposition}
Now, we prove that the property of tangency between the a web and a Levi-flat hypersurface is preserved by lifting to the projectivized tangent bundle. 
\begin{proposition}\label{tange}
Let $\mathcal{W}$ be a holomorphic $d$-web on a compact complex surface $X$ tangent to a real-analytic Levi-flat hypersurface $M\subset X$. Then 
there exists a real-analytic Levi-flat hypersurface $N\subset W$ tangent to $\F_W$.
\end{proposition}
\begin{proof}
Without loss of generality, we can assume that $M$ is irreducible.
By Lemma \ref{lift}, there exists a real-analytic subvariety $N$ of real dimension three in $\mathbb{P}T^{*}X$ such that $M'\subset N$, where $M'$ is the lifting of the regular part of $M$ under $\pi:\mathbb{P}T^{*}X\to X$. Let $W\subset X$ be the complex subvariety associated with $\W$.
In order not to load notations, we denote $N$ by $N\cap W$. Therefore, $N$ is a real-analytic hypersurface in $W$, of course, $N$ is Levi-flat, because the tangency condition between $\W$ and $M$ implies the tangency between $\F_W$ and $N$ in some dense open subset of $W$. It suffices to assert that $N$ is a real-analytic Levi-flat hypersurface in $W$, by Proposition \ref{piece}.
\end{proof}
\par We will use the following result given by Beltr\'an--Fern\'andez-P\'erez--Neciosup \cite[Theorem 1]{andres}.

\begin{theorem}\label{second}
Let $\mathcal{F}$ be a holomorphic foliation on a compact complex surface $X$ tangent to an irreducible real-analytic Levi-flat hypersurface $M \subset X$. Suppose that the self-intersection $C\cdot C>0$, where $C\subset M$ is an irreducible compact complex curve invariant by $\mathcal{F}$, then there exists a dicritical singularity $p \in \sing (\mathcal{F})\cap C$ such that $\mathcal{F}$ has a non-constant meromorphic first integral at $p$.
\end{theorem}
\par Finally, we state an important criteria to find meromorphic first integrals for webs due to Pan \cite[Lemme 4.2]{Pan}.
\begin{proposition}\label{meromor}
Let $\mathcal{W}$ be a holomorphic $d$-web on a compact complex surface $X$ and let $W$ be the complex subvariety associated to $\W$.
Let $\sigma:Y\to W$ be a resolution of singularities of $W$. Denote by $\tilde{\F}$ the lifting of the foliation $\F_W$ under $\sigma$. Suppose that the restriction of $\tilde{\F}$
to each connected component of $Y$ admits a meromorphic first integral. Then $\W$ admits a multiple-valued meromorphic first integral on $X$.
\end{proposition}

\section{Theorem \ref{first}}\label{primer}
We can now formulate and prove our first result.
\begin{maintheorem}\label{first}
Let $\W=(W,\pi_W,\F_W,\Gamma_W)$ be a holomorphic $d$-web on a compact complex surface $X$ tangent to a real-analytic Levi-flat hypersurface  $M\subset X$ such that the discrimimant curve $\Delta(\W)$ is contained in $M$. Suppose that the critical curve $\Gamma_W$ of $\W$ is invariant by the foliation $\mathcal{F}_W$ and the self-intersection $\Gamma_{W}\cdot\Gamma_W>0$. Then there exists an irreducible component $B$ of $W$ such that $\W$ is non-dicritical on $B$. Moreover, there exists a point $\zeta\in\Delta(\W)$ such that $\mathcal{W}$ admits a local multiple-valued meromorphic first integral at $\zeta$.
\end{maintheorem}
\begin{proof}
 After resolution of singularities, we can assume that $W$ is smooth. 
 By Proposition \ref{tange}, there exists a real-analytic Levi-flat hypersurface $N\subset W$ tangent to $\F_W$. Since $\Gamma_W$ is an invariant complex curve by $\F_{W}$, Camacho-Sad's formula implies 
$$\sum_{q\in\Gamma_W\cap Sing(\F_W)}CS_q(\F_W,\Gamma_W)=\Gamma_W\cdot\Gamma_W>0.$$
Therefore, $\Gamma_W\cap Sing(\F_W)\neq \emptyset$. Let $q\in\Gamma_W\cap Sing(\F_W)$ and fix an irreducible component $B$ of $W$ such that $q\in B$. Because $\F_{W}$ has singularities along $B$, it follows that $\W$ is non-dicritical on $B$ (in the sense of the Definition \ref{dicri}). 
\par On the other hand, since the discrimimant curve $\Delta(\W)$ is contained in $M$ we have $\Gamma_W\subset N$.
Now, notice that we are in the conditions of Theorem \ref{second}, because
$\F_W$ is tangent to $N$, $\Gamma_W\subset N$ and $\Gamma_W\cdot\Gamma_W>0$, then the foliation $\F_W$ has a local non-constant meromorphic first integral at $q\in \Gamma_W$. Hence, $\W$ admits a multiple-valued meromorphic first integral at $\pi(q)=\zeta\in\Delta(\W)$ by Proposition \ref{meromor}. This suffices to prove the theorem.
\end{proof}
Next, we will give an example in $\mathbb{P}^2$ where the conditions of Theorem \ref{first} are satisfied.
\begin{example}
Let $\mathcal{W}$ be the 2-web in $\mathbb{P}^2$ defined in affine coordinates $(x,y)\in\mathbb{C}^2$ by
\[\omega=ydy^2+xdy dx\]
In this chart, $\mathcal{W}=(xdx+ydy)\boxtimes dy$. Then the local leaves of $\mathcal{W}$ are given by superposition of the curves  $x^2+y^2=c_1$ and $y=c_2$, where $c_1,c_2\in\mathbb{C}$. The discriminant curve is $\Delta(\W)=\{x=0\}$.

Consider the real-analytic Levi-flat hypersurface \[M=\{\Re e(x^2+y^2)=0\}\cup\{\Re e(y)=0\}.\]
Observe that $M$ extends to a Levi-flat hypersurface on $\mathbb{P}^2$ and clearly $\Delta(\W)\subset M$. Furthermore, it is tangent to $\W$. 
The subvariety $W$ associated to $\mathcal{W}$ is given by
\[W=\{(x,y,p):yp^2+xp=0\},\] where $p=\frac{dy}{dx}$. The criminant curve  is
 $\Gamma_{W}=\Gamma_1\cup\Gamma_2$, where $\Gamma_1=\{p=x=0\}$ and $\Gamma_2=\{x=y=0\}$ 
are irreducible components of $\Gamma_W$. The foliation $\mathcal{F}_{W}$ is defined by
\[\alpha= \left.(dy-pdx)\right|_{W}= -xdp+p(p^2+1) dx,\]
and it admits a meromorphic first integral 
\[f(x,p)=\frac{x^2(p^2+1)}{p^2}.\] It is not difficult to see that $\W$ is non-dicritical on $\Gamma_2$.

\end{example}
\section{Theorem \ref{second2}}\label{segundo}
We have thus state and prove our second result.
\begin{secondtheorem}\label{second2}
Let $\mathcal{W}=(W,\pi_W,\F_W,\Gamma_W)$ be a smooth dicritical holomorphic $d$-web on a compact complex surface $X$ tangent to a real analytic Levi-flat hypersurface $M\subset X$ such that the discriminant curve $\Delta(\W)\subset M$. Suppose that $X\setminus\Delta(\W)$ is a Stein manifold and the critical curve $\Gamma_W$ is not invariant by $\mathcal{F}_{W}$.
Then $\W$ admits a global multiple-valued meromorphic first integral. 
\end{secondtheorem}
\begin{proof}
Since $\W$ is smooth and dicritical we have $W$ is a (smooth) complex surface and 
$\mathcal{F}_{W}$ is a regular foliation on $W$. By Proposition \ref{tange}, there exists a real-analytic Levi-flat hypersurface $N\subset W$ tangent to $\F_W$. Because $\Delta(\W)\subset M$, we have $\Gamma_W\subset N$. Let $\Tilde{\zeta} \in \Gamma_W \subset N$. Using Cartan's Theorem \cite[Theorem IV]{cartan} and the fact that $\Gamma_W$ is not invariant by $\mathcal{F}_W$, we have local holomorphic coordinates $(x,y)\in U\subset W$, centered at $\tilde{\zeta}$, such that \[\Gamma_W\cap U=\{y=0\}\,\,\,\,\,\,\text{and}\,\,\,\,\,\, N\cap U=\{\Im m(x)=0\},\]
so $\left.\mathcal{F}_W\right|_U$ is given by $dx=0$. Because $\mathcal{F}_W$ is a regular foliation, for every $\Tilde{\zeta} \in \Gamma_W$ there is only one leaf $L_{\Tilde{\zeta}}$ of $\mathcal{F}_W$ passing by $\Tilde{\zeta}$. Thus, $\mathcal{F}_W$ is locally defined by a fibration $\Tilde{f}:U_{\Tilde{\zeta}} \rightarrow \Gamma_W$ in some open neighborhood $U_{\Tilde{\zeta}} \subset W$ of each point $\Tilde{\zeta} \in \Gamma_W$ and therefore $\tilde{f}(L_{\tilde{\zeta}})=L_{\tilde{\zeta}}\cap \Gamma_W.$
\par Consequently, via collage, we find an open neighborhood $U \subset W$ of $\Gamma_W$ such that $\mathcal{F}_W$ is defined by a fibration $\Tilde{f}: U \rightarrow \Gamma_W$. By assumption, $X\setminus\Delta(\W)$ is a Stein manifold, therefore $W\setminus\Gamma_W$ is also a Stein manifold, since 
$\pi_W: W\setminus\Gamma_W\rightarrow X\setminus\Delta(\W)$ is a $d$-fold covering. Thus $\Tilde{f}$ extends to a global fibration $f: W \rightarrow \Gamma_W$. Essentially, this fibration works as a projection of leaves of $\mathcal{F}_{W}$ onto the curve $\Gamma_W$. In fact, since $W\setminus\Gamma_W$ does not contain complex analytic sets of positive dimensions, every leaf of $\mathcal{F}_{W}$ intersects $\Gamma_W$ at a single point, because $\mathcal{F}_{W}$ is a regular foliation. Finally, we pick any rational map $g: \Gamma_W \dashrightarrow \PP^1$ and $h:=g\circ f$. Hence, $h$ defines a meromorphic first integral for $\F_W$, and $\mathcal{W}$ admits a multiple-valued meromorphic first integral, by Proposition \ref{meromor}. This concludes the proof of Theorem \ref{second2}.
\end{proof}
We present the following example, where Theorem \ref{second2} applies. 
\begin{example}\label{example2}
Consider the family of hyperplanes $H_t:=\{X+tY+t^2Z=0; t \in \C\}$ in coordinates $(X,Y,Z) \in \C^3$. By eliminating $t$ in the system $G_t=dG_t=0$ (that is, calculating the resulting $R[G_t,dG_t]$, we obtain a 2-web germ, which extends to the global 2-web $\mathcal{W}$ of degree 0 in $\PP^2$, whose leaves are the hyperplanes $H_t$, with $t \in \C$. Furthermore, $\mathcal{W}$ is tangent to the Levi-flat hypersurface defined by elimination of $t$ of the system $G_t=\overline{G_t}=0$, whose leaves of the Levi foliation are the hyperplanes $H_t$, with $t \in \R$.

In homogeneous coordinates $[X:Y:Z]$, the 2-web $\mathcal{W}$ is given by the 2-symmetric form defined by
\[\omega=\det\left(\begin{array}{cccc}
    X & Y & Z & 0 \\
    0 & X & Y & Z\\
    dX & dY & dZ & 0\\
    0 & dX & dY & dZ
\end{array}
\right),
\]
that is, 
\[\omega=Z^2(dX)^2+XZ(dY)^2+X^2(dZ)^2+ (Y^2-2XZ)dXdZ-XYdYdZ-YZdXdY.\]
In the chart $Z=1$, we denote by $(x,y)$ the affine coordinates of $\PP^2$, so we have $$\mathcal{W}: (dx)^2-ydxdy+x(dy) ^2=0.$$ The associated subvariety
$W=\{xp^2-yp+1=0\}$ is smooth, here $p=\frac{dy}{dx}$. The criminal curve is defined by 
$\Gamma_W=\{xp^2-1=y-2xp=0\}$, is also smooth. Also, $\mathcal{W}$ is dicritical because $\mathcal{F}_W$ is regular on $\Gamma_W$. Furthermore, $\Gamma_W$ is not invariant by $\mathcal{F}_W$. By construction, $\mathcal{W}$ admits the global multiple-valued meromorphic first integral
$$P([X:Y:Z],t)=t^2Z+tY+X.$$
\end{example}
\section{Webs on the complex projective space}\label{projective-space}
Let $\W=[\omega]\in\mathbb{P}H^{0}(\mathbb{P}^n,Sym^{d}\Omega_{\mathbb{P}^n}^{1}\otimes\mathcal{N})$ be a holomorphic $d$-web on $\mathbb{P}^n$. The \textit{degree of $\W$} is defined as the number of tangencies, counted with multiplicities, of $\W$ with a line not everywhere tangent to $\W$, (see for instance \cite{JV}). Recall that every line-bundle $\mathcal{L}$ on $\mathbb{P}^n$ is an integral multiple of $\mathcal{O}_{\mathbb{P}^n}(1)$  and consequently one can write $\mathcal{L}=\mathcal{O}_{\mathbb{P}^n}(\deg{\mathcal{L}})$. Using the identity $Sym^{d}\Omega_{\mathbb{P}^{1}}^1=\mathcal{O}_{\mathbb{P}^1}(-2d)$ we obtain
\[\deg(\W)=\deg(\mathcal{N})-2d.\]
\par Let $\check{\mathbb{P}}^{n}$ denote the projective space parametrizing hyperplanes in $\mathbb{P}^n$.
\subsection{Algebraic webs}
Given a projective curve $C\subset\mathbb{P}^n$ of degree $d$ and a hyperplane 
$H_0\in\check{\mathbb{P}}^{n}$ intersecting $C$ transversely, there exists a natural germ of quasi-smooth $d$-web $\W_C:=\W_C(H_0)$ on $(\check{\mathbb{P}}^n,H_0)$ defined by the submersions $p_1,\ldots,p_d:(\check{\mathbb{P}}^n,H_0)\to C$ which describe the intersections of $H\in(\check{\mathbb{P}}^n,H_0)$ with $C$. The $d$-webs of the form $\W_C(H_0)$ for some reduced projective curve $C$ and some transverse hyperplane $H_0$ are classically called \textit{algebraic webs}. For more details about the construction of algebraic webs, we refer the reader to \cite[Section 1.1.3]{JV}.
\par Let $\mathcal{I}\subset\mathbb{P}\times\check{\mathbb{P}}^n$ be the incidence variety, that is 
\[\mathcal{I}=\{(p,H)\in\mathbb{P}\times\check{\mathbb{P}}^n:p\in H\}.\]
We denote $\pi$ and $\check{\pi}$ the natural projections from $\mathcal{I}$ to $\mathbb{P}^n$ and $\check{\mathbb{P}}^n$. According to \cite[Proposition 1.4.1]{JV} the incidence variety $\mathcal{I}$ is naturally isomorphic to $\mathbb{P}T^{*}\mathbb{P}^n$
and also to $\mathbb{P}T^{*}\check{\mathbb{P}}^n$. Using this identification of $\mathcal{I}$ with $\mathbb{P}T^{*}\check{\mathbb{P}}^n$ one defines for every projective curve $C\subset\mathbb{P}^n$ its dual web $\W_C$ as the one defined by variety $\pi^{-1}(C)\subset\mathbb{P}T^{*}\check{\mathbb{P}}^n$ and it can be seen as a multi-section of $\check{\pi}:\mathbb{P}T^{*}\check{\mathbb{P}}^n\to \check{\mathbb{P}}^n$. It is evident that the germification of this global web at a generic point $H_0\in \check{\mathbb{P}}^n$ coincides with the germ of web $\W_C(H_0)$ defined as above. 
\par We will use the following characterization of algebraic webs given in \cite[Proposition 1.4.2]{JV}.
\begin{proposition}\label{degw}
Let $C\subset\mathbb{P}^n$ be a projective curve of degree $d$. Then $\mathcal{W}_C$ is a $d$-web of degree zero on $\check{\mathbb{P}}^n$. Recipocally, if $\mathcal{W}$ is a $d$-web of degree zero in $\Check{\mathbb{P}}^n$, then there exists $C \subset \mathbb{P}^n$, a projective curve of degree $d$, such that $\mathcal{W}=\mathcal{W}_C$.
\end{proposition}
In order to prove the main result of this section, we will enunciate the following result.
\begin{theorem}\cite[Shafikov-Sukhov]{Russians}\label{russians}
Let $M \subset U$ be an irreducible real-analytic Levi-flat hypersurface in a domain $U \subset \C^n$, with $n \geq 2$ and $0 \in \overline{M_{reg}}$. Assume that at least one of the following conditions is true:
\begin{enumerate}
     \item $0 \in M$ is non-dicritical singularity.
     \item $M$ is a real algebraic hypersurface.
\end{enumerate}
Then there is a neighborhood $U'$ of the origin and a holomorphic $d$-web $\mathcal{W}$ in $U'$ such that $\mathcal{W}$ extends the Levi foliation of $M$. Furthermore, $\mathcal{W}$ admits a multiple-valued meromorphic first integral in $U'$.
\end{theorem}
\par Real-analytic hypersurfaces induced by an irreducible real-analytic curve in the Grassmannian $G(n+1,n)$ were studied in \cite{Lebl}, in particular, when these hypersurfaces have dicritical singularities, Lebl \cite{Lebl} proved the following important result. 
\begin{theorem} \label{lebl2}
Suppose that $M \subset \mathbb{P}^n$ is a real-analytic hypersurface induced by an irreducible real-analytic curve $S\subset G(n+1,n)$, and suppose that $M$ has a dicritical singular point. Then there is an integer $1 \leq \ell < n$, homogeneous coordinates $[z_0: \cdots: z_n]$, and a bihomogeneous real polynomial $\rho(z_0, \cdots, z_{\ell}, \z_0, \cdots, \z_{\ell})$, such that
\begin{enumerate}
     \item $M$ is the zero set of $\rho$ in $\mathbb{P}^n$. In particular, $M$ is algebraic,
     \item $\rho$ defines a Levi-flat algebraic hypersurface $X \subset \mathbb{P}^{\ell}$, with non-dicritical singularities, induced by a curve in $G(\ell +1, \ell )$.
\end{enumerate}
\end{theorem}
We can now formulate our third result. 
\begin{thirdtheorem}\label{third}
Let $M \subset \mathbb{P}^n$, $n\geq 2$, be a real-analytic Levi-flat hypersurface induced by an irreducible real-analytic curve  $S \subset G(n+1,n)$. Then there exists a projective curve  $C\subset\mathbb{P}^n$ such that $S \subset C$ and the Levi foliation $\mathcal{L}_M$ extends to algebraic web $\mathcal{W}_C$. Moreover, $\mathcal{W}_C$ admits a global multiple-valued meromorphic first integral.
\end{thirdtheorem}
\begin{proof}
Let $p\in\sing(M)$. We have two possibilities for $p$ in $M$: either $p$ is a non-dicritical singularity; or $p$ is a dicritical singularity. If $p$ is a dicritical singularity, then $M$ is a real-algebraic hypersurface by Theorem \ref{lebl2}, item $(1)$.
Without loss of generality, we can assume that $p$ is the origin in inhomogeneous coordinates given by $z_0=1$. Let $z=(z_1, \cdots, z_n)$ and denote by $M_0$ the germ of $M$ at $0 \in \C^n$.
Therefore, in a small neighborhood of $p$, $M_0$ satisfies one of the following statements:
\begin{enumerate}
    \item $0 \in \C^n$ is non-dicritical singularity of $M_0$;
    \item $M_0$ is algebraic.
\end{enumerate}
Both cases satisfy the hypotheses of Theorem \ref{russians}, so that we have, in both cases, that there exists a neighborhood $U$ of the origin and a singular holomorphic $d$-web $\mathcal{W}_p$ in $U$ such that $\W_p$ extends the Levi foliation $\mathcal{L}_{M_{0}}$. Furthermore, $\mathcal{W}_p$ admits a multiple-valued meromorphic first integral in $U$.
%\[f(\tilde{z},t)=t^d+a_1(\tilde{z})t^{d-1}+\cdots+a_d(\tilde{z}),\]
%where $a_k$ are germs of holomorphic functions at $(\C^n,0)$, see for instance \cite[pg. 493]{Russians}. 
%\par Let us now work in homogeneous coordinates $z=[z_0:\ldots:z_n]$, then we get a small neighborhood $U$ of $p$ in $\PP^n$ and a holomorphic $d$-web on $U$, also denoted by $\mathcal{W}_p$ tangent to $M\cap U$. Moreover, $\mathcal{W}_p$ admits a multiple-valued meromorphic first integral in $U$
%\[F(z,t)=t^d+A_1(z)t^{d-1}+\cdots+A_d(z),\]
%where $z \in \C^{n+1}-\{0\}$ and $A_k(z)\in\mathcal{O}(U)$. Note that, $\mathcal{W}_p$ is given by the elimination of $t$ in the system of equations $F=\frac{\partial{F}}{\partial{z}}=0$ (i.e., via resultant of $F$ and $\frac{\partial{F}}{\partial{z}})$, and the tangency condition between $\W_p$ and $M$ implies that the set
\par Assume that $M_0$ is defined by 
\[M_0=\{F(z,\bar{z})=0\},\] 
where $F$ is a real-analytic function in $U$.  
According to \cite[Proposition 4.1]{Russians}, there exists a complex line $A\subset\C^n$ such that $A\cap Q_0=\{0\}$, $A\not\subset \sing(M_0)$ and for every $q\in (M_{0})_{reg}$, there exists a point $w\in A$ such that the leaf $L_q$ of $\mathcal{L}_{M_0}$ is contained in $Q_w$. Denote by $\mathbb{D}_{\epsilon}$ the open disk in $\C$ with center $0$ and radius $\epsilon$. Let $\mathbb{D}_{\epsilon}\ni t\to w(\bar{t})=(w_1(\bar{t}),\ldots,w_n(\bar{t}))$, $w(0)=0$, be an anti $\C$-linear parametrization of the complex line $A$, then, by \cite[pg. 492]{Russians}, we obtain that the connected components of the Segre variety 
$$Y=\{(z,t) \in U\times \mathbb{D}_{\epsilon}: F(z,\overline{w(\bar{t})})=0\}$$
are leaves of $\W_p$. 
\par Now, fix $t \in \mathbb{D}_{\epsilon}$, and set $F_t:=\{z \in U: F(z,\overline{w(\bar{t})})=0\}$. Let $\sigma: \C^{n+1}-\{0\} \longrightarrow \mathbb{P}^n$ be the natural projection. We claim that the complex subvarieties $$H_t:=\{Z \in \C^{n+1}:\ Z=0\ \text{or}\ \sigma(Z) \in F_t\}$$ are complex semialgebraic subsets (i.e., every $H_t$ is contained in an algebraic variety). The process to verify this statement is the same as used in the proof of Chow's Theorem, as we will see below. First, we observe that if $Z \in H_t$ then $\lambda Z \in H_t$ for all $\lambda \in \C^*$, hence it follows that $0 \in \overline{H_t}$. By Remmert-Stein Theorem, $\overline{H_t}$ is a complex subvariety of $\C^{n+1}$. Furthermore, since $\mathcal{O}_{n+1}$ is Notherian, the ideal of $\overline{H_t}$ in $\mathcal{O}_{n+1}$, which we denote by $I(\overline{H_t})$, is finitely generated. Assume that $h^{t}_1,\cdots,h^{t}_m$ are generators of $I(\overline{H_t})$, defined in a common open subset $V \subset \C^{n+1}$. For every $k=1,\cdots,m$, consider the Taylor series of $h^{t}_k$ at $0\in\C^{n+1}$
\[h^{t}_k(Z)=\sum_{\mu\geq 0} h^{t}_{k\mu}(Z),\]
where $h^{t}_{k\mu}$ is a homogeneous polynomial of degree $\mu$. Let $$V^{t}_j=\{Z \in V: \ h^{t}_{k\mu}(Z)=0, \; \forall \mu \leq j \; \mbox{and} \; k=1, \cdots,m \}.$$
Clearly $V^{t}_{j+1} \subset V^{t}_j$, for every $j \geq 0$. Thus, we obtain an ascending chain of ideals in $\mathcal{O}_{n+1}$
\[ I(V^{t}_0) \subset \cdots \subset I(V^{t}_j) \subset I(V^{t}_{j+1}) \subset \cdots.\]
Again since $\mathcal{O}_{n+1}$ is Notherian, there is $j_0$ such that $I(V^{t}_j)=I(V^{t}_{j_0})$ for every $j \geq j_0$. Hence, $V^{t}_j=V^{t}_{j_0}$ for every $j \geq j_0$ and 
\[\overline{H_t} \cap U=\{z \in U: \ h^{t}_{k\mu}(z)=0, \; \forall \mu \leq j_0 \; \mbox{and} \; k=1, \cdots,m \}.\]
This concludes that $H_t$ is a complex semialgebraic subset. Now consider the sets
\[W_t=\{Z \in \C^{n+1}: \ h^{t}_{k\mu}(Z)=0, \; \forall \mu \leq j_0, \; \mbox{and} \; k=1, \cdots,m \}.\]
The germ of $W_t$ are equal to the germ of $\overline{H_t}$ at $0\in\C^{n+1}$. Thus, the projections $\sigma(W_t)$ generate a finite family of algebraic subvarieties in $\PP^n$ and, via the resultant (by elimination of the variable $t$), we get a global $d$-web $\mathcal{W}$ in $\PP^n$, whose germ at $p$ coincides with $\W_p$. By construction, this web admits a global multiple-valued meromorphic first integral and it extends the foliation of Levi $\mathcal{L}_M$. Since the leaves of $\mathcal{L}_M$ are hyperplanes in $\PP^n$, it follows that the closure of the leaves of $\mathcal{W}$ are hyperplane too and consequently $\W$ has degree 0. Therefore, $\mathcal{W}$ is algebraic, that is, there exists a projective curve $C \subset \mathbb{P}^n$ of degree $d$ such that $\mathcal{W}=\mathcal{ W}_C$, by Proposition \ref{degw}. Furthermore, since $M$ is induced by the irreducible real analytic curve $S$, we conclude that $S \subset C$. This completes the proof of Theorem \ref{third}.
\end{proof}
\par It is easy to build examples where Theorem \ref{third} applies, for instance, in Example \ref{example2}, the web presented is algebraic and it is tangent to a Levi-flat hypersurface induced by a real-algebraic curve in $\mathbb{P}^1$.

\end{document}